\documentclass{amsart}
\usepackage{amssymb}
\usepackage{epsfig}  	
\usepackage{epic,eepic}

\usepackage{hyperref}
\newtheorem{thm}{Theorem}[section]
\newtheorem{Prop}[thm]{Proposition}
\newtheorem{Lemma}[thm]{Lemma}
\newtheorem{Cor}[thm]{Corollary}

\theoremstyle{definition}
\newtheorem{defn}[thm]{Definition}

\newtheorem*{Constructionoutline}{Construction Outline}
\newtheorem*{Construction}{Construction}
\newtheorem*{construction/proof of prop}{Construction/Proof of Proposition}

\theoremstyle{remark}

\newtheorem{Fact}[thm]{Remark}

\numberwithin{equation}{section}

\DeclareMathOperator{\RHH}{\mathbb{H}H} 
\begin{document}
	
	\title{Cohomological Operators on Quotients by Exact Zero Divisors}
	\author{Andrew Windle}
	\address{Department of Mathematics, Rockhurst University, 
		Kansas City, MO 64110}
	\email{andrew.windle@rockhurst.edu}

	\begin{abstract}
		Let $S$ be a commutative ring, $x, y \in S$ a pair of exact zero divisors, and $R = S/(x)$.  Let $F$ be a complex of free $R$-modules.  In this paper we explicitly compute cohomological operators of $R$ over $S$ by constructing endomorphisms of $F$.  We consider some properties of these cohomological operators, as well as provide an example in which these cohomological operators act non-trivially.
	\end{abstract}
	
	\maketitle
	
	
	\section{Introduction}
	Let $R$ be a commutative ring, and $M$, $N$ a pair of $R$-modules.  Whenever $R$ is realized as the quotient $S/I \cong R$ of another commutative ring $S$, the graded $R$-module $\text{Ext}_R(M,N)$ acquires the extra structure of a graded module over $\RHH(R/S)$, the derived Hochschild cohomology ring of $R$ over $S$.  In particular, each element of $\RHH(R/S)$ induces what we call a \textit{cohomological operator} on $R$-modules;  that is, it induces a family of elements $g_M \in \text{Ext}_R^d(M,M)$, where $d$ is the degree of the element, for each $R$-module $M$.  These elements are compatible in the following sense:  if we consider the action of $g_M$ and $g_N$ on $h \in \text{Ext}_R(M,N)$, we have that $$g_N \cdot h = (-1)^{|h|d} h \cdot g_M.$$ The only case where cohomology operators are well understood is in the case when $I$ is a complete intersection ideal.
		
	The first attempt to study $R$-modules through cohomological operators was done by Gulliksen \cite{GulOp}.  When $I$ is generated by a regular sequence, he showed that if $\text{Ext}_S(M,N)$ is finitely generated in each cohomological degree (for instance, when $M$ and $N$ are finitely generated $S$-modules) and $\text{Ext}^n_S(M,N) = 0$ for $n \gg 0$, then $\text{Ext}_R(M,N)$ is a finitely generated module over the ring $R[ \chi_1, \cdots \chi_c]$ where each $\chi_j$ is of cohomological degree 2.  The ring $R[ \chi_1, \cdots, \chi_c ]$ is sometimes referred to as the ring of cohomological operators of $R$ over $S$, since in this case $\RHH(R/S) \cong R[ \chi_1, \cdots, \chi_c]$.  This case of cohomological operators has been studied from a variety of perspectives, including \cite{mehta}, \cite{Eis}, \cite{AvOp}. The work of Avramov and Sun \cite{AvSun} shows that these various constructions of cohomological operators have similar actions on $\text{Ext}_R(M,N)$ and $\text{Tor}^R(M,N)$, differing by at most a sign.  Using these various developments of cohomological operators, results such as the development of complete intersection dimension \cite{cid}, support varieties for complete intersection rings \cite{AvB2}, and a structure theory for free resolutions over complete intersections \cite{MFSforCIs} were generated.
	
	In this paper, we consider the next simplest case for cohomological operators:  when $R$ is a quotient of $S$ by a member of an exact pair of zero divisors.  We construct explict cohomology operators of degree $2$ and $3$ and establish some of their properties.  In a future paper, we will show these operators really do arise from the action of $\RHH(R/S)$.
	
	\begin{defn} Let $S$ be a commutative ring.  We say that two elements $(x, y) \in S$ form an \textit{exact pair of zero divisors} if $\text{ann}_S(x) = (y)$ and $\text{ann}_S(y) = (x)$.  Equivalently, two elements $(x,y)$ form an exact pair of zero divisors if the unbounded complex 
		$$ \cdots \stackrel{x}{\rightarrow} S \stackrel{y}{\rightarrow} S \stackrel{x}{\rightarrow} S \stackrel{y}{\rightarrow} S \stackrel{x}{\rightarrow} \cdots $$
		is an exact sequence.  When referring to a single element from the pair, we call that element an {exact zero divisor}.
	\end{defn}
	
	Our goal is to consider the following construction, which we outline here and provide below:
	\begin{Constructionoutline} 
		Let $S$ be a commutative ring, $(x,y)$ a pair of exact zero divisors in $S$, and define $R := S/(x)$.  Let $(F, d^F)$ be a complex of free $R$-modules.		
		\begin{enumerate}
			\item Choose an arbitrary lifting of the complex $F$ to a sequence of maps of free $S$-modules, denoted $(\tilde{F}, \tilde{d})$.  The maps $\tilde{d}^2$ need not be zero (see definition \ref{liftingdef} below).
			\item For a choice of $(\tilde{F}, \tilde{d})$ there exists a degree $-2$ $S$-linear mapping $\tilde{\psi}: \tilde{F} \rightarrow \tilde{F}$ such that $x\tilde{\psi} = \tilde{d}^2$.
			\item For a choice of $\tilde{\psi}$, there exists a degree $-3$ $S$-linear mapping $\tilde{\varphi}:\tilde{F} \rightarrow \tilde{F}$ such that $\tilde{d} \circ \tilde{\psi} - \tilde{\psi} \circ \tilde{d} = y\tilde{\varphi}$.
		\end{enumerate}

		Define $\varphi := \tilde{\varphi} \otimes_S R$ and $\psi_z := z(\tilde{\psi} \otimes_S R)$ for all $z \in \text{ann}_R(y)$.
	\end{Constructionoutline}

	We prove the following theorem over the course of the next two sections.
	
	\begin{thm} \label{bigtheorem1}
		Let $S$ be a commutative ring, $(x,y)$ a pair of exact zero divisors in $S$, and define $R := S/(x)$.  Let $(F, d^F)$ be a complex of free $R$-modules.
		Using the notation above, $\varphi$ and $\psi_z$ are degree $-3$ and $-2$ chain maps such that:
		\begin{enumerate}
			\item The chain maps are independent of all choices made, up to homotopy.
			\item The construction is natural, up to homotopy, for maps of chain complexes.
			\item The maps $\varphi$, $\psi_z$ commute, up to homotopy, for all $z \in \text{ann}_R(y)$.
			\item The maps $\psi_z$, $\psi_{z'}$ commute exactly for all $z, z' \in \text{ann}_R(y)$.
			\item $2\varphi^2$ is homotopic to the zero map.
		\end{enumerate} 
	\end{thm}
	
	\begin{Fact}
		In Proposition \ref{nontrivialopex}, we provide an example in which the chain maps of the theorem above are not homotopic to the zero map.
	\end{Fact}
	
	We call the collection of chain maps above \textit{cohomological operators} of $R$ over $S$.  The following corollary follows immediately from the theorem:
	
	\begin{Cor} \label{EXTeltscor}
		Let $M$ be an $R$ module with free resolution $F_M$.  There exist degree $3$ and $2$ elements of $\text{Ext}_R(M,M)$ arising from the theorem applied to $F_M$ that are independent of all choices made and commute exactly.  If $\dfrac{1}{2} \in S$, these elements generate a graded commutative subring of $\text{Ext}_R(M,M)$.
	\end{Cor}
	
	\begin{Fact}
		We suspect that the actions of $\text{Ext}_R(M,M)$ and $\text{Ext}_R(N,N)$ on $\text{Ext}_R(M,N)$ and $\text{Tor}^R(M,N)$ are compatible up to the appropriate sign, but do not prove that in this paper.
	\end{Fact}
	
	\begin{Fact}
		In the case that $x = y$, we recover a special case of Eisenbud's construction in \cite[\S 1]{Eis}, since in this setting $I/I^2$ is free as an $R$-module.
	\end{Fact}
	
	In section 2, we provide the details of the construction outlined before the theorem and show that $\varphi$ and $\psi_z$ are chain maps.  We provide additional properties of the cohomological operators in the case that $S$ is a graded ring and $x, y$ are homogeneous elements of $S$.  Section 3 contains the proof of Theorem \ref{bigtheorem1}.  Finally, in section 4 we provide an example where $\varphi$ and $\psi_z$ are not homotopic to the zero map.

	\section{Construction of Cohomology Operators}
	
	\begin{defn}  \label{liftingdef}  A \textit{lifting} of a complex $(F,d^F)$ of free $R$-modules is any choice of free $S$-modules and maps $(\tilde{F}, \tilde{d})$ such that $\tilde{F}_j \otimes_S R = F_j$ and $\tilde{d}_j \otimes_S R = d_j$ for all $j$.  The definition does not assume that $\tilde{d}^2 = 0$.
	\end{defn}

	\begin{Construction} \label{operatorsconstruction}
		
		Let $(F,d^F)$ be a homologically indexed complex of free $R$-modules. Choose an arbitrary lifting $(\tilde{F}, \tilde{d})$ of the complex $(F,d^F)$. 
		We can concretely describe the lifting process in the following way for complexes of finitely generated free modules:  for each free $R$-module, replace that module with a free $S$-module of the same rank.  By choosing bases, represent each map $d^F_i$ with a matrix having entries in $R$.  For each of those entries, choose a lifting of them to the ring $S$.  The maps $\tilde{d}_i$ are represented by the matrices with the lifted entries to $S$.
		
		Since $F$ is a complex, $d^F_{i-1} \circ d^F_i =0$ for all $i$, but the same cannot be said for $\tilde{d}$.  Instead, since $\tilde{F} \otimes_S R = F$, the compositions $\tilde{d}_{i-1} \circ \tilde{d}_i \text{ are congruent to } 0 \text{  modulo  } x$.  
		Using that each of the $\tilde{F}_i$ is a free $S$-module, and that the multiplication by $x$ mapping $\tilde{F}_{i-2} \stackrel{\cdot x}{\rightarrow} x \tilde{F}_{i-2}$ is surjective, there exists a (non-unique) map of homological degree $-2$ which we denote by $\tilde{\psi}_i: \tilde{F}_i \rightarrow \tilde{F}_{i-2}$ such that $x\tilde{\psi}_i = \tilde{d}_{i-1} \circ \tilde{d}_{i}.$  
		As before, in the case that each $F_i$ is a finitely generated free module, we can even more explicitly describe the situation:  the entries of the matrices representing $\tilde{d}_{i-1} \circ \tilde{d}_{i}$ must all be divisible by the element $x \in S$, since tensoring with $R$ these compositions become zero.  Therefore, the entries of $\tilde{\psi}_i$ can be chosen to be elements of $S$ which when multiplied by $x$ give the appropriate entries of $\tilde{d}_{i-1} \circ \tilde{d}_{i}$.
		We denote the collection of chosen maps by $\tilde{\psi} := \{\tilde{\psi}_i\}_{i \in \mathbb{Z}}$.  To simplify notation, as is done with the differentials of a complex, we sometimes use $\tilde{\psi}$ to represent a single map in the collection by suppressing subscripts.  We will use the collection $\tilde{\psi}$ to find a family of degree 2 cohomological operators.
		
		Suppressing all sub and superscripts, we have that \[ x(\tilde{d}\tilde{\psi} - \tilde{\psi}\tilde{d})
		= \tilde{d}(x\tilde{\psi}) - (x\tilde{\psi})\tilde{d} 
		= \tilde{d}^ 3 - \tilde{d}^3 = 0. \]
		Therefore, we see that the element $x$ annihilates the map $\tilde{d}\tilde{\psi} - \tilde{\psi}\tilde{d}$.  Using that $\text{ann}_S(x) = (y)$, that each of the $\tilde{F}_i$ is a free $S$-module, and that $\tilde{F}_{i-3} \stackrel{\cdot y}{\rightarrow} y \tilde{F}_{i-3}$ is a surjection, there exists a (non-unique) homological degree $-3$ map $\tilde{\varphi}_i: \tilde{F}_i \rightarrow \tilde{F}_{i-3}$ such that $\tilde{d}\tilde{\psi} - \tilde{\psi}\tilde{d} = y \tilde{\varphi}$ for all appropriate choices of indices.  Again, we can describe the possible choices for the matrix representations of $\tilde{\varphi}$ in the case that $F$ is a complex of finitely generated free modules:  the entries of $\tilde{\varphi}$ are any element of $S$ which when multiplied by $y$ gives the appropriate entry of $\tilde{d}\tilde{\psi} - \tilde{\psi}\tilde{d}$. We denote the collection of chosen maps by $\tilde{\varphi} := \{ \tilde{\varphi}_i \}_{i \in \mathbb{Z}}$ and simplify notation by suppressing subscripts where helpful.  We will use $\tilde{\varphi}$ to find a degree 3 cohomological operator.  
		
	\end{Construction}
	
	Let $z \in \text{ann}_R(y)$.  Define  $\psi_z := z \cdot (\tilde{\psi} \otimes_S R) := \{ z \cdot (\tilde{\psi}_i \otimes_S R) \}_{i \in \mathbb{Z}}$ and $\varphi := \tilde{\varphi} \otimes_S R := \{\tilde{\varphi}_i \otimes_S R \}_{i \in \mathbb{Z}}$.
	
	\begin{Prop}  \label{chainmapprop}
		In the outlined construction above, for any complex of free $R$-modules $F$, the maps $\varphi$ and $\psi_z$ are chain maps for all $z \in \text{ann}_R(y)$.
	\end{Prop}
	
	\begin{proof}
		Let $z \in \text{ann}_R(y) \subset R$.  Using that $\tilde{d}\tilde{\psi} - \tilde{\psi}\tilde{d} = y \tilde{\varphi}$ and that multiplication by $z$ commutes with all maps, by tensoring with $R$ we have that $d (\psi_z) - (\psi_z) d = zy\varphi = 0$, making $\psi_z$ a degree $-2$ chain map on $F$.
		
		For the following we use the notation $[f, g] := f \circ g - (-1)^{|f||g|} g \circ f$ to simplify computations. 
		
		To see that $\varphi$ is a chain map of degree $-3$, since $\tilde{\psi}$ has even degree, we have: \\
		\[ [\tilde{d},[\tilde{d},\tilde{\psi}]] = [\tilde{d}, y\tilde{\varphi}] = y [\tilde{d}, \tilde{\varphi}]. \]
		On the other hand, we also have \[ [\tilde{d},[\tilde{d},\tilde{\psi}]] = [\tilde{d}, \tilde{d}\tilde{\psi} - \tilde{\psi}\tilde{d}] = \tilde{d}(\tilde{d}\tilde{\psi} - \tilde{\psi}\tilde{d}) + (\tilde{d}\tilde{\psi} - \tilde{\psi}\tilde{d})\tilde{d} = \tilde{d}^2\tilde{\psi} - \tilde{d}\tilde{\psi}\tilde{d} + \tilde{d}\tilde{\psi}\tilde{d} - \tilde{\psi}\tilde{d}^2 = x\tilde{\psi}^2 - x\tilde{\psi}^2 = 0. \]
		Putting these together, we have that $0 =y[\tilde{d},\tilde{\varphi}]$, and so this commutator is annihilated by the element $y$. 
		Therefore, by the same arguments as before, since $\text{ann}_S(y) = (x)$, all of the $S$-modules $\tilde{F}_i$ are free, and the mapping $\tilde{F}_{i-4} \stackrel{\cdot x}{\rightarrow} x \tilde{F}_{i-4}$ is surjective for all $i$, there is a (non-unique) homological degree $-4$ map $\tilde{\rho}$ such that $[\tilde{d}, \tilde{\varphi}] = x\tilde{\rho}$, or $\tilde{d}\tilde{\varphi} + \tilde{\varphi}\tilde{d} = x\tilde{\rho}$. 
		Tensoring with $R$, we have that $d \varphi + \varphi d = 0$, making $\varphi$ a degree $-3$ chain map on $F$. 
		
	\end{proof}

	\subsection{Graded Construction of Cohomology Operators}
	In this subsection, let $S$ be a $\mathbb{Z}$-graded ring.   We restrict our study of exact zero divisors to homogeneous elements $x,y$ and complexes of graded free modules $(F, d^F)$ over the graded ring $R = S/(x)$.  In this setting, we obtain further information about the cohomological operators from the main section.
	
	\begin{defn}  \label{gradedliftingdef}  A \textit{homogeneous lifting} of the complex $(F,d^F)$ is any choice of graded free $S$-modules and internal degree zero maps $(\tilde{F}, \tilde{d})$ such that $\tilde{F}_j \otimes_S R = F_j$ and $\tilde{d}_j \otimes_S R = d_j$ for all $j$.
	\end{defn}
	
	First, we adjust the construction from the main section to fit within the graded setting:
	
	\begin{Constructionoutline}
		Let $S$ be a $\mathbb{Z}$-graded ring, $(x,y)$ a pair of homogeneous exact zero divisors in $S$, and define $R := S/(x)$.  Let $(F, d^F)$ be a complex of graded free $R$-modules.
		\begin{enumerate}
			\item Choose an arbitrary homogeneous lifting of the complex $F$ to $S$, denoted $(\tilde{F}, \tilde{d})$.
			\item For a choice of $(\tilde{F}, \tilde{d})$ there exists a $S$-linear mapping $\tilde{\psi}: \tilde{F} \rightarrow \tilde{F}$ of internal degree $-\text{deg}(x)$ and homological degree $-2$ such that $x\tilde{\psi} = \tilde{d}^2$.
			\item For any given choice of $\tilde{\psi}$, there exists a $S$-linear mapping $\tilde{\varphi}:\tilde{F} \rightarrow \tilde{F}$ of internal degree $-(\text{deg}(x) + \text{deg}(y))$ and homological degree $-3$ such that $\tilde{d} \circ \tilde{\psi} - \tilde{\psi} \circ \tilde{d} = y\tilde{\varphi}$.
		\end{enumerate}

		Define $\varphi := \tilde{\varphi} \otimes R$ and $\psi_z := z(\tilde{\psi} \otimes R)$ for all homogeneous $z \in \text{ann}_R(y)$.
	\end{Constructionoutline}

	\begin{Prop}
		Let $S$, $R$, and $F$ be as defined as above.  Let $z \in \text{ann}_R(y)$ be homogeneous.  Then $\psi_z$ is a chain map of internal degree $\text{deg}(z) + \text{deg}(\psi) = \text{deg}(z) - \text{deg}(x)$ and $\varphi$ is a chain map of internal degree $- (\text{deg}(y) + \text{deg}(x))$.
	\end{Prop}
	
	\begin{construction/proof of prop}
		Consider a complex of graded free $R$-modules 
		$$ F: \cdots \rightarrow F_{i+1} \stackrel{d_{i+1}}{\rightarrow} F_i \stackrel{d_i}{\rightarrow} F_{i-1} \rightarrow \cdots$$
		where each $d_i$ is an internal degree zero module homomorphism: that is, $\text{deg}(t) = \text{deg}(d_i(t))$ for all homogeneous $t \in F_i$, for all $i \geq 0$.
		
		As in the previous section, we can choose a homogeneous lifting to a sequence of $S$-modules which we denote by $\tilde{F}$ with mappings $\tilde{d}$.  Since $\tilde{d}^2 \equiv 0$ mod $x$, there is a homogeneous homological degree $-2$ map $\tilde{\psi}$ satisfying the equation $x\tilde{\psi} = \tilde{d}^2$.  Since we have chosen $\tilde{d}$ to be an internal degree zero map, we must have that the internal degree of the mapping $\tilde{\psi}$ is $-\text{deg}(x)$.  It follows that we also have that $\text{deg}(\psi) = -\text{deg}(x)$.
		
		Following the same procedure as in the main section, we can also define a homogeneous map $\tilde{\varphi}$ satisfying the equation $y\tilde{\varphi} = \tilde{d}\tilde{\psi} - \tilde{\psi}\tilde{d}$.  This mapping is of homological degree $-3$.  Since the internal degree of $\tilde{d}$ is zero and the internal degree of $\tilde{\psi}$ is $-\text{deg}(x)$, we have that $\text{deg}(y\tilde{\varphi}) = \text{deg}(y) + \text{deg}(\tilde{\varphi}) =  -\text{deg}(x)$, and so as a consequence, $\text{deg}(\varphi) = -(\text{deg}(x) + \text{deg}(y))$. 
		
	\end{construction/proof of prop} 
	
	In the graded examples that we consider in section 5, we will use the internal degrees of these maps to make computations easier.
	
	\section{Properties of Cohomology Operators}
	The cohomological operators we built in the previous section depended upon several choices:  we chose a lifting of the complex $F$, and made choices for both $\tilde{\psi}$ and $\tilde{\varphi}$.  We now show that the cohomological operators $\varphi$ and $\psi_z$ are well-defined up to homotopy for all $z \in \text{ann}_R(y)$.  When considering the actions of the classes of these operators in homology, this makes the action completely well-defined regardless of any choices made.  We additionally will show that any two cohomological operators commute up to homotopy, and that the map $2\varphi^2$ is homotopic to zero.
	
	The following proposition is similar to \cite[1.3]{Eis} for cohomological operators in the case that $I$ is generated by a regular  sequence:
	
	\begin{Prop} \label{naturalityprop}
		Let $(F,d)$ and $(G,d')$ be two complexes of finitely generated free $R$-modules, $f: F \rightarrow G$ a homological degree $k$ morphism of complexes. Let $\tilde{F}$, $\tilde{\psi}$, $\tilde{\varphi}$ be chosen as in section 3 from the complex $F$.  Let $\tilde{G}, \tilde{\psi'}, \tilde{\varphi'}$ be chosen as in section 3 from the complex $G$.  Then $\psi'_z \circ f$ is homotopic to  $f \circ \psi_z$ and $\varphi' \circ f$ is homotopic to  $(-1)^{k} f \circ \varphi$. Moreover, in the graded case, if $f$ is a homogeneous map of complexes, then up to homogeneous homotopy the cohomological operators commute with $f$.
	\end{Prop}
	
	\begin{proof}
		Since $f: F \rightarrow G$ is a homomorphism of complexes of homological degree $k$, we have  that $d' f - (-1)^k f d = 0$. 
		Choose a sequence of maps $\tilde{\theta}_j: \tilde{F}_j \rightarrow \tilde{G}_{j + k - 1}$ so that $\tilde{d'}\tilde{f} - (-1)^k \tilde{f} \tilde{d} = x \tilde{\theta}$ for all appropriately chosen indices.  
		Rearranging, we have that $\tilde{d'}\tilde{f} = (-1)^k \tilde{f}\tilde{d} + x\tilde{\theta}$, and $\tilde{f}\tilde{d} = (-1)^k \tilde{d'}\tilde{f} - (-1)^k x \tilde{\theta}$, which we will use in future computations.
		We start by showing that $\psi'_z \circ f$ is homotopic to $f \circ \psi_z$.  We have that 
		\begin{align*}
		x\tilde{f}\tilde{\psi} &= \tilde{f}(x\tilde{\psi}) \\ 
		&= \tilde{f}\tilde{d}^2 = [(-1)^k \tilde{d'}\tilde{f} - (-1)^k x \tilde{\theta}]\tilde{d}  \\ 
		&= (-1)^k \tilde{d'}\tilde{f}\tilde{d} - (-1)^k x \tilde{\theta}\tilde{d} \\  
		&= (-1)^k\tilde{d'}[(-1)^k \tilde{d'}\tilde{f} - (-1)^k x \tilde{\theta}] - (-1)^k x\tilde{\theta}\tilde{d} \\ 
		&= \tilde{d'}^2\tilde{f} -x\tilde{d'}\tilde{\theta} - (-1)^k x \tilde{\theta}\tilde{d} \\  
		&= x\tilde{\psi'}\tilde{f} -x\tilde{d'}\tilde{\theta} - (-1)^k x \tilde{\theta}\tilde{d}. \end{align*}

		Rearranging the equality, we have that  $$x(\tilde{f}\tilde{\psi} - \tilde{\psi'}\tilde{f} + \tilde{d'}\tilde{\theta} + (-1)^k \tilde{\theta}\tilde{d}) = 0,$$ so following the same procedure as in the proof of Proposition \ref{chainmapprop}, there is a degree $k-2$ map $\vartheta$ such that $$\tilde{f}\tilde{\psi} - \tilde{\psi'}\tilde{f} + \tilde{d'}\tilde{\theta} + (-1)^k \tilde{\theta}\tilde{d} = y\tilde{\vartheta}.$$
		Therefore, we have that $$\tilde{f}\tilde{\psi} = \tilde{\psi'}\tilde{f} + \tilde{d'}(-\tilde{\theta}) - (-1)^{k-1}(-\tilde{\theta})\tilde{d} + y\tilde{\vartheta}.$$ 
		Alternatively, we can write $$\tilde{\psi'}\tilde{f} = \tilde{f}\tilde{\psi} + \tilde{d'}\tilde{\theta} + (-1)^k\tilde{\theta}\tilde{d} - y\tilde{\vartheta}.$$
		We will need both forms of this equation for future computations as well.
		Tensoring this equality with $R$, we have that $$f\psi - \psi'f = d'(-\theta) - (-1)^{k-1}(-\theta)d + y\vartheta.$$  Let $z \in \text{ann}_R(y)$.  
		Multiplying by $z$, we get that $$f(\psi_z) - (\psi'_z)f = d'(-z\theta) - (-1)^{k-1}(-z\theta)d,$$ and thus $f \circ (\psi_z)$ is homotopic to $(\psi'_z) \circ f$.
		
		To show that $\varphi' \circ f$ is homotopic to $(-1)^k f \circ \varphi$, consider the following:
		\begin{align*}
		y\tilde{\varphi'}\tilde{f} &= [\tilde{d'}\tilde{\psi'} - \tilde{\psi'}\tilde{d'}]\tilde{f} \\
		&= \tilde{d'}\tilde{\psi'}\tilde{f} - \tilde{\psi'}[x\tilde{\theta} + (-1)^k \tilde{f}\tilde{d}] \\
		&= \tilde{d'}[\tilde{f}\tilde{\psi} + \tilde{d'}\tilde{\theta} + (-1)^k\tilde{\theta}\tilde{d} - y\tilde{\vartheta}] - x\tilde{\psi'}\tilde{\theta} - (-1)^k[\tilde{f}\tilde{\psi} + \tilde{d'}\tilde{\theta} + (-1)^k\tilde{\theta}\tilde{d} - y\tilde{\vartheta}]\tilde{d} \\
		&= [x\tilde{\theta} + (-1)^k \tilde{f}\tilde{d}]\tilde{\psi} + x\tilde{\psi'}\tilde{\theta} - y \tilde{d'}\tilde{\vartheta} - x\tilde{\psi'}\tilde{\theta} - (-1)^k \tilde{f}\tilde{\psi}\tilde{d}-x\tilde{\theta}\tilde{\psi} + (-1)^k y \tilde{\vartheta}\tilde{d} \\
		&= (-1)^k[\tilde{f}(\tilde{d}\tilde{\psi} - \tilde{\psi}\tilde{d})] - y\tilde{d'}\tilde{\vartheta} + (-1)^ky\tilde{\vartheta}\tilde{d} \\
		&= (-1)^k\tilde{f}(y\tilde{\varphi}) - y\tilde{d'}\tilde{\vartheta} + (-1)^ky\tilde{\vartheta}\tilde{d}
		\end{align*}
		
		By rearranging, we have that $$y[\tilde{\varphi'}\tilde{f} - (-1)^k \tilde{f}\tilde{\varphi} + \tilde{d'}\tilde{\vartheta} - (-1)^k\tilde{\vartheta}\tilde{d}] = 0,$$ and so by the same technique as in the proof of Proposition \ref{chainmapprop} $$\tilde{\varphi'}\tilde{f} - (-1)^k \tilde{f}\tilde{\varphi} + \tilde{d'}\tilde{\vartheta} - (-1)^k\tilde{\vartheta}\tilde{d} = x\tilde{\Upsilon}$$ for some degree $k-3$ mapping $\tilde{\Upsilon}$.
		Tensoring with $R$, we obtain the equation $$\varphi'f - (-1)^k f \varphi = d'(-\vartheta) - (-1)^{k-2} (-\vartheta)d.$$
		
		This shows that $\varphi' \circ f$ is homotopic to $(-1)^k f \circ \varphi$, as wanted.
		
		By assuming that $f$ is a homogeneous map, following the same arguments provided above one can see that the homotopies $\theta$ and $\vartheta$ are also homogeneous.
		
	\end{proof}
	
	\begin{Cor} \label{well-definedcor}
		Let $z \in \text{ann}_R(y)$.  The cohomological operators $\psi_z$ and $\varphi$ are independent of all choices made, up to homotopy.  In the graded case, for homogeneous $z \in \text{ann}_R(y)$, the cohomological operators are independent of all choices made up to homogeneous homotopy.
	\end{Cor}
	
	\begin{proof}
		In Proposition \ref{naturalityprop}, take $G = F$ and let $f = id^F$.  This shows that for any two choices of liftings $\tilde{F}, \tilde{F}'$ and choices of maps $\tilde{\psi}, \tilde{\psi}', \tilde{\varphi}, \text{ and } \tilde{\varphi}'$, the map $\psi_z$ is homotopic to $\psi'_z$ and $\varphi$ is homotopic to $\varphi'$.  The graded case follows immediately from the graded cases of \ref{naturalityprop} as well.
	\end{proof} 
	
	\begin{Cor} \label{commutationcor}
		The cohomological operators constructed in section three commute with each other up to homotopy.  That is, for any elements $z, z' \in \text{ann}_R(y)$ we have that $ \psi_z \circ \psi_{z'} = \psi_{z'} \circ  \psi_z$ and $ \psi_z \circ \varphi$ is homotopic to $\varphi \circ  \psi_z$.  Furthermore, $2 \varphi^2$ is homotopic to 0.  In the graded case, for $z, z'$ homogeneous, these operators commute up to homogeneous homotopy.
	\end{Cor}
	
	\begin{proof}
		It is clear that any two of the degree 2 cohomological operators commute directly with one another, since $\psi_z \circ \psi_{z'} = z \cdot z' \psi^2 = z' \cdot z \psi^2 =  \psi_{z'} \circ \psi_z.$
		For the remaining two facts we again apply Proposition \ref{naturalityprop}. Take $G = F$ and let $f = \psi_{z}$.  Then we have that $\psi_{z} \circ \varphi$ is homotopic to $\varphi \circ \psi_z$.  Likewise, if we take $G = F$ and let $f = \varphi$, then we have that $\varphi \circ \varphi + \varphi \circ \varphi = 2 \varphi^2$ is homotopic to 0.
	\end{proof}
	
	\begin{Cor}
		Let $M$ be an $R$-module, and let $F_M$, $F'_M$ be two free resolutions of $M$.  Let $z \in \text{ann}_R(y)$, $\text{cls}(\psi_z)$ and $\text{cls}(\varphi)$ be the classes in homology of the cohomological operators constructed from $F_M$, and $\text{cls}(\psi'_z), \text{cls}(\varphi')$ be the classes in homology of the cohomological operators constructed from $F'_M$.  Then $\text{cls}(\psi_z) = \text{cls}(\psi'_z)$ and $\text{cls}(\varphi) = \text{cls}(\varphi')$ in $\text{Ext}_R(M,M)$.
	\end{Cor}
	
	\begin{proof}
		Let $f: F_M \rightarrow F'_M$ be a homotopy equivalence between the two free resolutions of $M$.  Then there exists a second homotopy equivalence $g: F'_M \rightarrow F_M$ such that $f \circ g$ is homotopic to the identity on $F'_M$ and $g \circ f$ is homotopic to the identity on $F_M$.  We can compute $\text{Ext}_R(M,M)$ through $\text{H}^{\bullet}\text{Hom}_R(F_M, F_M)$ or through $\text{H}^{\bullet}\text{Hom}_R(F'_M, F'_M)$.  The isomorphism between these sends $\text{cls}(\alpha)$ to $\text{cls}(g \circ \alpha \circ f)$ for all $\alpha \in \text{Hom}_R(F'_M, F'_M)$.  We know by Proposition \ref{naturalityprop} that $\psi'_z \circ f$ is homotopic to $ f \circ \psi_z$ and $\varphi' \circ f$ is homotopic to $f \circ \varphi$.  By composing with $g$ on the left, in homology, we have that $\text{cls}(\psi_z)$ is equal to $\text{cls}(g \circ \psi'_z \circ f)$ and $\text{cls}(\varphi)$ is equal to $\text{cls}(g \circ \varphi' \circ f)$.  Therefore, under the isomorphism of $\text{Ext}_R(M,M)$, we identify $\text{cls}(\psi_z)$ and $\text{cls}(\psi'_z)$ and identify $\text{cls}(\varphi)$ with $\text{cls}(\varphi')$.
	\end{proof}
	
	We end the section by noting that the contents of Theorem \ref{bigtheorem1} have now been proven:  Proposition \ref{chainmapprop} shows us that for all $z \in \text{ann}_R(y)$, the maps $\psi_z$ and $\phi$ are chain maps.  Part 2 of the theorem is shown in Proposition \ref{naturalityprop}.  Part 1 is shown in Corollary \ref{well-definedcor}, and parts 3, 4, and 5 are shown in Corollary \ref{commutationcor}.
	
	\section{An Example of Non-Vanishing Cohomology Operators}
	In this section, we use the graded cohomological operators from section 3 to produce an example where some of the cohomological operators constructed are not homotopic to zero.
	
	\begin{Prop} \label{nontrivialopex}
		Let $S = \mathbb{Q}[x, y, z, w, t]/(x^4, y^4, w^4, z^4, x^2y^2, y^2w^2, z^2w^2, xt, zt, wt)$.
		The elements $f = x^2 + y^2 + z^2 + w^2$ and $g = x^2 + y^2 - z^2 - w^2$ form an exact pair of zero divisors in $S$.  Put $R = S/(f)$.  Then $t \in \text{ann}_R(g)$, and the operators $\varphi$ and $\psi_{t}$ are not homotopic to the zero map on the minimal $R$-free resolution of $M = R/(y)$.  In other words, the elements $\text{cls}(\varphi)$ and $\text{cls}(\psi_t)$ of $\text{Ext}_R(M,M)$ are non-zero.
	\end{Prop}
	
	The proof of Proposition \ref{nontrivialopex} will make up the remainder of this section.
	
	Using computer algebra software such as Macaulay2, one can verify that $\text{ann}_S(f) = (g)$ and $\text{ann}_S(g) = (f)$, making the elements an exact pair of zero divisors.
	
	To see that the cohomological operators developed in the previous sections act non-trivially, we resolve the $R$-module $R/(y)$.  
	The start of the minimal graded free resolution of $R/(y)$ is given by:
	$$ \cdots \rightarrow R(-4)^4 \oplus R(-5)^6 \oplus R(-6)^6 \stackrel{d_3}{\rightarrow} R(-3) \oplus R(-4)^3 \stackrel{d_2}{\rightarrow} R(-1) \stackrel{d_1}{\rightarrow} R \rightarrow 0$$
	where the differentials can be expressed using the following matrices:
	$$d_1 = [y], d_2 = [yt \quad yw^2 \quad yz^2 \quad y^3] $$
	
	$d_3 = 
	\begin{bmatrix}
	w & z & y & x & 0 & 0 & 0 & 0 & 0 & 0 & 0 & 0 & 0 & 0 & 0 & 0 \\
	0 & 0 & 0 & 0 & t & 0 & 0 & y & 0 & 0 & w^2 & z^2 & 0 & 0 & 0 & 0 \\
	0 & 0 & 0 & 0 & 0 & t & 0 & 0 & y & 0 & 0 & 0 & w^2 & 0 & z^2 & 0 \\
	0 & 0 & 0 & 0 & 0 & 0 & t & 0 & 0 & y & 0 & 0 & 0 & w^2 & 0 & z^2
	\end{bmatrix}
	$
	
	We start by lifting the $R$- free resolution to a sequence of $S$-module maps.  Since the operators are independent of the choice of lifting, take $\tilde{d}_1, \tilde{d}_2, \tilde{d}_3$ to be the same matrices as $d_1, d_2, d_3$, with elements regarded as being in $S$ instead of in $R$.  
	
	We have that 
	$$ \tilde{d}_1 \circ \tilde{d}_2 = [y^2t \quad 0 \quad y^2z^2 \quad 0]. $$  The construction from section 1 of the chapter tells us that the entries of this matrix should be divisible by the element $f = x^2 + y^2 + z^2 + w^2$, and the other factor forms one choice of entries for the matrix $\tilde{\psi}_2$.  Using the relations in the ring $S$, we see that $$\tilde{\psi}_2 = [t \quad 0 \quad z^2 - x^2 + w^2 \quad 0]$$ is one such choice, and that $\tilde{d}_1 \circ \tilde{d}_2 = f \tilde{\psi}_2$, as wanted.  
	
	We also have that $$\tilde{d}_2 \circ \tilde{d}_3 = [0 \quad 0 \quad y^2t \quad 0 \quad 0 \quad 0 \quad y^3t \quad 0 \quad y^2z^2 \quad 0 \quad 0 \quad 0 \quad 0 \quad 0 \quad 0 \quad y^3z^2].$$  Using the same arguments as for the previously considered composition, we have
	$$ \tilde{\psi}_3 = [0 \quad 0 \quad t \quad 0 \quad 0 \quad 0 \quad yt \quad 0 \quad z^2 - x^2 + w^2 \quad 0 \quad 0 \quad 0 \quad 0 \quad 0 \quad 0 \quad y(z^2 - x^2 + w^2)]. $$
	
	Using the now computed first two non-zero maps in $\tilde{\psi}$, we can compute the first non-zero map in $\tilde{\varphi}$.
	
	The results from section 3 guarantee that the difference $\tilde{d}_1 \circ \tilde{\psi}_3 - \tilde{\psi}_2 \circ \tilde{d}_3$ must be represented by a matrix in which every entry is divisible by the element $g = x^2 + y^2 - z^2 - w^2$.  
	
	First, by direct computation, we have that $\tilde{d}_1 \circ \tilde{\psi}_3 - \tilde{\psi}_2 \circ \tilde{d}_3$ is represented by the matrix  
	$$[0 \quad 0 \quad 0 \quad 0 \quad 0 \quad 0 \quad y^2t \quad 0 \quad 0 \quad 0 \quad 0 \quad 0 \quad x^2w^2 \quad 0 \quad z^2x^2 \quad y^2z^2]. $$
	
	Using the relations in $S$, we see that each of these elements is divisible by $g$ and one choice for the entries of $\tilde{\varphi}_3$ is the following matrix:
	
	$$[0 \quad 0 \quad 0 \quad 0 \quad 0 \quad 0 \quad t \quad 0 \quad 0 \quad 0 \quad 0 \quad 0 \quad w^2 \quad 0 \quad y^2 + z^2 \quad -y^2].$$
	
	By taking the tensor product with $R$, we obtain maps $\psi_2, \psi_3, \varphi_3$ which have the same entries as before, except regarded as elements of $R$ instead of $S$.  
	
	\begin{Lemma} The degree three cohomological operator $\varphi$ is not homotopic to zero on the resolution of $R/(y)$.
	\end{Lemma}
	
	\begin{proof}
		It is enough to show that $\varphi_3$ is not homotopic to the zero map on the resolution of $R/(y)$.
		
		By way of contradiction, suppose there is a homotopy $\theta$ with maps $\theta_2: R(-3) \oplus R(-4)^3 \rightarrow R$ and $\theta_3:R(-4)^4 \oplus R(-5)^6 \oplus R(-6)^6 \rightarrow R(-1)$ such that $d_1 \theta_3 + \theta_2 d_3 = \varphi_3$.  Represent the homotopy maps with matrices: \\
		
		$\theta_3 = [a_1 \quad \cdots \quad a_{16}]$ and $\theta_2 = [b_1 \quad b_2 \quad b_3 \quad b_4]$.
		
		We know that the differentials have internal degree zero, and that the mapping $\varphi$ has internal degree $-4$, since $\text{deg}(f) = 2 = \text{deg}(g)$.  Therefore, the maps $\theta_2, \theta_3$ must each have internal degree $-4$ as well.  This gives us a restriction on the elements $a_i$ and $b_j$ above:  We know that $\theta_2: R(-3) \oplus R(-4)^3 \rightarrow R$, and so we must have that $b_1 = 0$ and that $ |b_j| = 0$ for $j = 2, 3, 4$.  Likewise, we have that $\theta_3: R(-4)^4 \oplus R(-5)^6 \oplus R(-6)^6 \rightarrow R(-1)$.  Therefore, we must have that $a_1 = a_2 = a_3 = a_4 = 0$, $|a_i| = 0$ for $5 \leq i \leq 10$ and $|a_i| = 1$ for $11 \leq i \leq 16$.  
		
		Equating the entries of the matrices $d_1 \theta_3 + \theta_2 d_3$ and $\varphi_3$, we must find solutions to the system of equations \begin{align*}
		y a_7 + b_4 t &= t \\
		y a_{13} + b_3 w^2 &= w^2 \\
		y a_{15} + b_3 z^2 &= y^2 + z^2 \\
		y a_{16} + b_4 z^2 &= -y^2.
		\end{align*}
		
		Using the degree considerations above, and the fact that a $\mathbb{Q}$-basis of $R_1$ is given by the elements $\{x, y, z, w, t \}$, we have that $a_7 = 0$ and $b_4 = 1$ by considering the first equation.  
		
		Now consider the last equation.  We have that $y a_{16} + z^2 = -y^2$.  By using the relation $x^2 + y^2 + z^2 + w^2 = 0$ in $R$, we can replace this equation with $y a_{16} = x^2 + w^2$.  Now, a $\mathbb{Q}$-basis of $R_2$ is given by the set of monomials $\{x^2, y^2, w^2, t^2, xy, xz, xw, yz, yw, zw, yt \}$  We know that $a_{16}$ is a degree 1 element of $R$, so we can rewrite this 
		equation as $y(c_1x + c_2y + c_3z + c_4w + c_5t) = x^2 + w^2$ using the basis for $R_1$ found above.  Thus, after rearranging this equation, we have that $c_1 yx + c_2 y^2 + c_3 yz + c_4 yw + c_5 yt - x^2 - w^2 = 0$, for which no solution exists. 
		
		Therefore, there is no homotopy $\theta$ that will make $\varphi$ homotopic to the zero operator.
		
	\end{proof}
	
	We now consider the degree 2 cohomological operators of $R$ over $S$.  Again, using software such as Macaulay 2, one can compute the annihilator of $g$ in $R$.  We have that $\text{ann}_R(g) = (t, y^2, z^2, w^2)$.
	
	\begin{Lemma}  The degree 2 cohomological operator $\psi_t$ is not homotopic to zero on the free resolution of $R/(y)$.  
	\end{Lemma}
	
	\begin{proof}
		As before, it suffices to show that $(\psi_t)_2$ is not homotopic to the zero map.  Now, by way of contradiction, suppose that there is a homotopy $\kappa$ with component maps $\kappa_1, \kappa_2$ such that $d_1 \kappa_2 + \kappa_1 d_2 = (\psi_t)_2$.  
		
		Represent the maps with matrices $\kappa_1 = [k_1]$ and $\kappa_2 = [l_1 \quad l_2 \quad l_3 \quad l_4]$.  
		Since the internal degrees of the differentials are zero and the internal degree of $\psi_t$ is $\text{deg}(t) + \text{deg}(\psi) = 1 + -2 = -1$, we must have that the internal degrees of these maps are also $-1$.  Now, we have that $\kappa_1: R(-1) \rightarrow R$, so $|k_1| = 0$.  We also have that $\kappa_2:R(-3) \oplus R(-4)^3 \rightarrow R(-1)$.  So we have that $|l_1| = 1$ and that $|l_i| = 2$ for $i = 2, 3, 4$.  
		Now, equating the coefficients of the matrices $d_1 \kappa_2 + \kappa_1 d_2$ and $(\psi_t)_2$, we have that $k_1 yt + y l_1 = t^2$.  Replacing $l_1$ with an arbitrary element of $R_1$ and rearranging, we have that $k_1 yt + y(c_1 x + c_2 y + c_3 z + c_4 w + c_5 t) - t^2 = (k_1 + c_5)yt + c_1 xy + c_2 y^2 + c_3 yz + c_4 yw - t^2 = 0$.  Again using the $\mathbb{Q}$-basis of $R_2$ found above, there are no solutions to this equation.  Therefore, no such homotopy exists.  
	\end{proof}
	
	By combining the previous two lemmas, we have shown that Proposition \ref{nontrivialopex} holds.
	
	\section{Further Directions}
	
	It can be proven that the cohomological operators described in Theorem 2.2 above arise from the canonical action of derived Hochschild cohomology of $R$ over $S$ in the case that $S$ is a characteristic zero ring.  The techniques involved in showing this include the use of $S$-linear maps known as \textit{acyclic twisting cochains}. 
	
	\begin{thm} \label{EPZDCH3}
		If $S$ is a commutative ring of characteristic zero containing a pair of exact zero divisors $(x,y)$, and $R := S/(x)$, then the derived Hochschild cohomology of $R$ over $S$ is given by: $$\RHH^j(R/S) = \begin{cases}
		R & \quad j = 0 \\
		0 & \quad j = 1, j < 0 \\
		\text{ann}_R(y) \cdot \gamma^{j/2}       & \quad j \text{ even }, j \geq 2 \\
		R/(y) \cdot \gamma^{(\frac{j-3}{2})}\eta  & \quad j \text{ odd }, j \geq 3 \\
		
		\end{cases}$$
		The ring structure, inherited as a quotient of a subring of $R[\gamma, \eta]$, is given in the following way:  If $z, z' \in \text{ann}_R(y)$, then $(z \gamma^k)(z' \gamma^i) = zz'\gamma^{i + j}$.  If $p \in R/(y)$, $z \in \text{ann}_R(y) \subset R$, then $(z\gamma^k)  (p\gamma^i\eta) = zp \gamma^{i + k} \eta$, where $zp$ is the action of $z \in R$ on $p \in R/(y)$.  For $p, q \in R/(y), (p \gamma^i\eta)(q \gamma^k\eta) = pq \gamma^{i+k}\eta^2 = 0$.  
		
		Moreover, if $M, N$ are $R$-modules, for $z\gamma \in \RHH^2(R/S)$, the action on $\text{Ext}_R(M,N)$ coincides with the action by $\psi_z$ as in Theorem \ref{bigtheorem1}.  Furthermore, the action of $\eta \in \RHH^3(R/S)$ coincides with the action by $\varphi$.
	\end{thm} 

For details, see \cite{Win17}.

	There are still a great deal of questions that would be interesting to answer about the cohomological operators above.  For instance, under what conditions (on the ring, module, exact zero divisor) do the cohomological operators vanish?  Are there classes of rings (or modules) under which the operators are always non-vanishing?  In particular, if $S$ is a complete intersection ring, then $R$ is also a complete intersection \cite{Andre85}.  As such, both $S$ and $R$, when thought of as quotients of some regular ring $Q$, have cohomological operators over $Q$.  Do the cohomological operators of $R$ over $S$ provide any new cohomological information?  
	
	Finally, in \cite{AvSun}, the many different ways of computing cohomological operators on quotients by a regular sequence are compared and shown to provide the same action up to sign.  It would be of interest to determine if these cohomological operators can similarly be constructed in alternative ways, and if so, if their actions agree up to sign, as in the regular sequence case.  
	
	\subsection*{Acknowledgements} The work in this paper comes from a portion of the author's dissertation at the University of Nebraska. The author would like to thank his advisor, Mark Walker, for his support in this project and for careful readings of early drafts of this paper.
	
	\bibliographystyle{amsalpha}
	\bibliography{ThesisBib2}
	
\end{document}